\documentclass{article}

\usepackage{amsmath,amsthm,bm}
\usepackage{authblk}
\usepackage{enumerate}
\usepackage{comment}
\usepackage{amssymb}
\usepackage{subcaption}
\usepackage{graphicx}
\usepackage{url}

\newtheorem{thm}{Theorem}[section]

\newtheorem{conj}[thm]{Conjecture}

\newtheorem{lem}[thm]{Lemma}

\makeatletter
\newcommand*{\house}[1]{%
  \mathord{%
    \mathpalette\@house{#1}%
  }%
}
\newcommand*{\@house}[2]{%
  \dimen@=\fontdimen8 %
      \ifx#1\scriptscriptstyle\scriptscriptfont
      \else\ifx#1\scriptstyle\scriptfont
      \else\textfont\fi\fi
      3 %
  \sbox0{%
    $#1%
      \vrule width\dimen@\relax
      \overline{%
        \kern2\dimen@
        \begingroup 
          #2%
        \endgroup
        \kern2\dimen@
      }%
      \vrule width\dimen@\relax
      \mathsurround=1.5\dimen@ 
    $%
  }%
  \ht0=\dimexpr\ht0-\dimen@\relax
  \dp0=\dimexpr\dp0+2\dimen@\relax
  \vbox{%
    \kern\dimen@ 
    \copy0 %
  }%
}

\begin{document}

\title{A Lower Bound for the Area of the Fundamental Region of a Binary Form}
\author{Jason Fang and Anton Mosunov}
\affil{University of Waterloo}


\date{}

\maketitle

\begin{abstract}
Let
$$
F(x, y) = \prod\limits_{k = 0}^{n - 1}(\delta_kx - \gamma_ky)
$$
be a binary form of degree $n \geq 1$, with complex coefficients, written as a product of $n$ linear forms in $\mathbb C[x, y]$. Let
$$
h_F = \prod\limits_{k = 0}^{n - 1}\sqrt{|\gamma_k|^2 + |\delta_k|^2}
$$
denote the \emph{height} of $F$ and let $A_F$ denote the area of the \emph{fundamental region} $\mathcal D_F$ of $F$:
$$
\mathcal D_F = \left\{(x, y) \in \mathbb R^2 \colon |F(x, y)| \leq 1\right\}.
$$
We prove that $h_F^{2/n}A_F \geq \left(2^{1 + (r/n)}\right)\pi$, where $r$ is the number of roots \mbox{of $F$} on the real projective line $\mathbb R\mathbb P^1$, counting multiplicity.
\end{abstract}

\section{Introduction}
Let
$$
F(x, y) = \prod\limits_{k = 0}^{n - 1}(\delta_kx - \gamma_ky)
$$
be a binary form of degree $n \geq 1$, with complex coefficients, written as a product of $n$ linear forms in $\mathbb C[x, y]$. Let
$$
h_F = \prod\limits_{k = 0}^{n - 1}\sqrt{|\gamma_k|^2 + |\delta_k|^2}
$$
denote the \emph{height} of $F$. We define the \emph{fundamental region} (also known as the \emph{fundamental domain}) $\mathcal D_F$ of $F$ as
$$
\mathcal D_F = \left\{(x, y) \in \mathbb R^2 \colon |F(x, y)| \leq 1\right\}.
$$
The area $A_F$ of the fundamental region $\mathcal D_F$ plays an important role in number theory. Consider the special case when $F \in \mathbb Z[x, y]$ is a binary form of degree $n \geq 3$ and nonzero discriminant $D_F$. In 1933, Mahler \cite{mahler} proved that the number of solutions $Z_F(m)$ to the \emph{Thue inequality}
$$
|F(x, y)| \leq m
$$
can be approximated by $A_Fm^{2/n}$, provided that $F$ is irreducible over the rationals (here $m$ is a non-negative integer). More precisely, he proved the existence of a positive real number $c_F$, which depends only on $F$, such that
$$
\left|Z_F(m) - A_Fm^{2/n}\right| \leq c_Fm^{1/(n - 1)}.
$$
In 2019, Stewart and Xiao \cite[Theorem 1.1]{stewart-xiao} proved that the number of integers $R_F(m)$ of absolute value at most $m$ which are represented by $F$ is asymptotic to $W_FA_Fm^{2/n}$, where $W_F$ is a positive rational number that can be computed in terms of the \emph{automorphism group} of $F$ (see \cite{stewart-xiao} for the definition of the automorphism group). Thus it is interesting to investigate how large and how small can $A_F$ be. Motivated by this question we succeeded in proving the following.

\begin{thm} \label{thm:main}
Let $F$ be a binary form with complex coefficients of degree $n \geq 1$, with exactly $r$ roots on the real projective line $\mathbb R\mathbb P^1$, counting multiplicity. Then
\begin{equation} \label{eq:main-inequality}
h_F^{2/n}A_F \geq \left(2^{1 + (r/n)}\right)\pi.
\end{equation}
\end{thm}

For complex numbers $x$ and $y$ with positive real parts, let
\begin{equation} \label{eq:beta}
B(x, y) = 2\int\limits_0^{\pi/2}(\sin \theta)^{2x - 1}(\cos \theta)^{2y - 1}d\theta
\end{equation}
denote the \emph{beta function}, written in its trigonometric form. The lower bound for $A_F$ given in Theorem \ref{thm:main} complements the upper bound for $A_F$ found by Bean \cite[Theorem 1]{bean94}. Bean's result asserts that for any binary form with complex coefficients having degree $n \geq 3$ and nonzero discriminant $D_F$ the following inequality holds:
$$
|D_F|^{1/n(n - 1)}A_F \leq 3B\left(\frac{1}{3}, \frac{1}{3}\right).
$$

One of the ingredients in our proof is Jensen's inequality, which states that
\begin{equation} \label{eq:jensen}
\varphi\left(\frac{1}{b - a}\int\limits_a^b f(x) dx\right) \geq \frac{1}{b - a}\int\limits_a^b\varphi(f(x))dx
\end{equation}
for every concave function $\varphi(x)$ and every real-valued function $f(x)$ that is Lebesgue-integrable on an interval $[a, b]$. Another important ingredient is the polar formula for the computation of $A_F$. Since the curve $|F(x, y)| = 1$ can be expressed in polar form as
$$
r(\theta) = \frac{1}{\left|F(\cos \theta, \sin \theta)\right|^{1/n}},
$$
from calculus we know that the area $\mathcal D_F$ bounded by this curve can be computed as follows:
\begin{equation} \label{eq:polar-formula}
A_F = \int\limits_0^{2\pi}\frac{1}{2}r^2(\theta)d\theta = \frac{1}{2}\int\limits_0^{2\pi}\frac{d\theta}{\left|F(\cos \theta, \sin \theta)\right|^{2/n}}.
\end{equation}
With this formula one can also notice that $A_{cF} = |c|^{-2/n}A_F$ for any nonzero complex number $c$. Since $h_{cF} = |c|h_F$, we can see that the quantity $h_F^{2/n}A_F$ is invariant under scaling by a nonzero complex number, i.e.,
$$
h_{cF}^{2/n}A_{cF} = \left(|c|h_F\right)^{2/n}\left(|c|^{-2/n}A_F\right) = h_F^{2/n}A_F.
$$

When $r = 0$ or $r = n$ the right-hand side of the inequality (\ref{eq:main-inequality}) becomes independent of both $r$ and $n$. In the case $r = n$ there does exist a family of binary forms $F_{n, n}$ such that
$$
\lim\limits_{n \rightarrow \infty}\left(h_{F_{n, n}}A_{F_{n, n}}\right) = 4\pi.
$$
This family was studied by Bean and Laugesen \cite{bean-laugesen} and a binary form $F_{n, n}$ is defined by
$$
F_{n, n}(x, y) = \prod\limits_{k = 1}^n\left(x\sin\frac{k\pi}{n} - y\cos\frac{k\pi}{n}\right).
$$
In this case, for $n \geq 3$ we have
$$
h_{F_{n, n}} = 1 \quad \text{and} \quad A_{F_{n, n}} = 4^{1 - 1/n}B\left(\frac{1}{2} - \frac{1}{n}, \frac{1}{2}\right).
$$
In the case $r = 0$ the family of binary forms $F_{n, 0}$ defined as
$$
F_{n, 0}(x, y) = (x - iy)^n
$$
satisfies $h_{F_{n, 0}}^{2/n}A_{F_{n, 0}} = 2\pi$ for every positive integer $n$. To see that this is the case, note that  $h_{F_{n, 0}} = 2^{n/2}$, and that
$$
|F(\cos \theta, \sin \theta)| = |(\cos \theta - i \sin \theta)^n| = \left|\exp\left(-\theta i\right)^n\right| = \left|\exp(-n\theta i)\right| = 1.
$$
Thus it follows from (\ref{eq:polar-formula}) that $A_{F_{n, 0}} = \pi$. The family $F_{n, 0}$ has a rather special property that the quantity $h_F^{2/\deg F} A_F$ attains its smallest value when $F = cF_{n, 0}$ for some $n \in \mathbb N$ and nonzero $c \in \mathbb C$.\footnote{Notice that there are other families that also attain the minimum. For example, for \mbox{even $n$} and nonzero $c \in \mathbb C$ the minimum is attained by the form $c(x^2 + y^2)^{n/2}$.} By combining the above two examples together, we define the family of binary forms
\begin{equation} \label{eq:Fnr}
F_{n, r}(x, y) = (x - iy)^{n - r}\prod\limits_{k = 1}^r\left(x\sin\frac{k\pi}{r} - y\cos\frac{k\pi}{r}\right)
\end{equation}
and conjecture the following.

\begin{conj} \label{conj:main}
Let $F$ be a binary form with complex coefficients of degree $n \geq 1$, with exactly $r$ roots on the real projective line $\mathbb R\mathbb P^1$, counting multiplicity. Then
$$
h_F^{2/n}A_F \geq h_{F_{n, r}}^{2/n}A_{F_{n, r}}.
$$
\end{conj}

For a $2 \times 2$ matrix $M = \left(\begin{smallmatrix}a & b\\c & d\end{smallmatrix}\right)$ define
$$
F_M(x, y) = F(ax + by, cx + dy).
$$
We say that two binary forms $F$ and $G$ are \emph{equivalent under $\operatorname{GL}_2(\mathbb R)$} \mbox{if and only if} there exists an invertible $2 \times 2$ matrix $M$, with real coefficients, such that $F_M = G$. Conjecture \ref{conj:main} naturally complements the following conjecture of Bean \cite[Conjecture 1]{bean95}.

\begin{conj}
The maximal value $M_n$ of the quantity $|D_F|^{1/n(n - 1)}A_F$ over the class of forms of degree $n$ with complex coefficients and nonzero discriminant $D_F$ is attained precisely when $F$ is a form which, up to multiplication by a complex number, is equivalent under $\operatorname{GL}_2(\mathbb R)$ to the form $F_{n, n}$. In particular,
$$
M_n = D_{F_{n, n}}^{1/n(n - 1)}A_{F_{n, n}}.
$$
\end{conj}

The paper is organized as follows: in Section \ref{sec:lemma} we prove two auxiliary lemmas, in Section \ref{sec:proof} we prove Theorem \ref{thm:main}, and in Section \ref{sec:Gnr} we compute the height and the area of the fundamental region of a binary form $F_{n, r}$ defined in (\ref{eq:Fnr}).

\section{Auxiliary Results} \label{sec:lemma}

In this section we prove two auxiliary lemmas.

\begin{lem} \label{lem:logsin}
$\displaystyle \int\limits_0^{\frac{\pi}{2}} \log(\sin(x))dx = -\frac{\pi}{2}\log 2$.
\end{lem}

\begin{proof}
%
Denote the left-hand side by $I$. Then
$$
I = \int\limits_0^{\frac{\pi}{2}}\log(\sin(x))dx = \int\limits_0^{\frac{\pi}{2}}\log\left(\cos\left(\frac{\pi}{2} - x\right)\right)dx = \int\limits_0^{\frac{\pi}{2}}\log(\cos(x))dx,
$$
so
\begin{align*}
2I & = \int\limits_0^{\frac{\pi}{2}}\log(\sin(x))dx + \int\limits_0^{\frac{\pi}{2}}\log(\cos(x))dx\\
& = \int\limits_0^{\frac{\pi}{2}}\log(\sin(x)\cos(x))dx\\
& = \int\limits_0^{\frac{\pi}{2}}\log\left(\frac{\sin(2x)}{2}\right)dx\\
& = \int\limits_0^{\frac{\pi}{2}}\log(\sin(2x))dx - \frac{\pi}{2}\log 2\\
& = \frac{1}{2}\int\limits_0^{\pi}\log(\sin x)dx - \frac{\pi}{2}\log 2\\
& = I - \frac{\pi}{2}\log 2.
\end{align*}
Hence $I = -\frac{\pi}{2}\log 2$. 
\end{proof}

\begin{lem} \label{lem:bound}
For every real number $\beta$,
$$
\int\limits_0^{\pi}\log((\sin\beta\cos\theta)^2 + (\cos\beta\sin\theta)^2)\ d\theta\leq -\pi\log 2.
$$
Furthermore, the equality is attained if and only if $\beta = \frac{\pi}{4} + \frac{\pi}{2}k$ for some $k \in \mathbb Z$.
\end{lem}

\begin{proof}
Let $I(\beta)$ denote the left-hand side. Then it follows from Jensen's inequality (\ref{eq:jensen}) with $\varphi(x) = \log x$ that
\begin{align*}
I(\beta) &\leq \pi\log\left(\frac{1}{\pi}\int\limits_0^\pi 
\left(\sin^2\beta\cos^2\theta+\cos^2\beta\sin^2\theta\right)\ d\theta\right)\\
&=\pi\log\left(\frac{1}{\pi}\sin^2\beta\int\limits_0^\pi \cos^2\theta\ d\theta +\cos^2\beta\int\limits_0^\pi\sin^2\theta\ d\theta\right)\\
&=\pi\log\left(\frac{1}{\pi}\sin^2\beta \cdot \frac{\pi}{2}+\cos^2\beta \cdot \frac{\pi}{2}\right)\\
&=\pi\log\left(\frac{1}{2}\right)\\
&=-\pi \log 2.
\end{align*}
Since the function $\log x$ is strictly convex, the equality holds if and only if
\begin{equation} \label{eq:func}
\log((\sin\beta\cos\theta)^2 + (\cos\beta\sin\theta)^2)
\end{equation}
is constant on $(0,\pi)$. This means that for any $\theta_1,\theta_2\in (0,\pi)$,
\begin{align*}
\log((\sin\beta\cos\theta_1)^2 + (\cos\beta\sin\theta_1)^2)
&=\log((\sin\beta\cos\theta_2)^2 + (\cos\beta\sin\theta_2)^2),\\
(\sin\beta\cos\theta_1)^2 + (\cos\beta\sin\theta_1)^2
&=(\sin\beta\cos\theta_2)^2 + (\cos\beta\sin\theta_2)^2.
\end{align*}
We can pick $\theta_1=\frac{\pi}{6}$ and $\theta_2=\frac{2\pi}{3}$ to obtain:
\begin{align*}
\left(\sin\beta \cdot \frac{\sqrt{3}}{2}\right)^2 + \left(\cos\beta \cdot \frac{1}{2}\right)^2
&=\left(\sin\beta \cdot -\frac{1}{2}\right)^2 + \left(\cos\beta \cdot \frac{\sqrt{3}}{2}\right)^2,\\
\frac{3}{4}\sin^2\beta + \frac{1}{4}\cos^2\beta
&=\frac{1}{4}\sin^2\beta + \frac{3}{4}\cos^2\beta,\\
\frac{1}{2}\sin^2\beta
&=\frac{1}{2}\cos^2\beta,\\
\sin^2\beta
&=\cos^2\beta.
\end{align*}
This happens if and only if $\beta=\frac{\pi}{4} + \frac{\pi}{2}k$ for some integer $k$. Plugging this value of $\beta$ into (\ref{eq:func}), we see that the resulting function is constant and equal to $-\log 2$. Hence $I\left(\frac{\pi}{4} + \frac{\pi}{2}k\right) = -\pi \log 2$ for every integer $k$.
\end{proof}


\section{Proof of Theorem \ref{thm:main}} \label{sec:proof}

Suppose that the binary form $F(x, y)$ has exactly $r$ roots $(\gamma_j \colon \delta_j)$ on $\mathbb R\mathbb P^1$, counting multiplicity. Write
$$
F(x, y) = \prod\limits_{j = 0}^{r - 1}(\delta_jx - \gamma_j y) \prod\limits_{k = r}^{n - 1}(\delta_k x - \gamma_k y),
$$
where $\delta_j, \gamma_j \in \mathbb R$ and $\delta_k, \gamma_k \in \mathbb C$. Notice that we can rewrite $F$ as follows:
$$
F(x, y) = h_F\prod\limits_{j = 0}^{r - 1}(\delta_j'x - \gamma_j' y) \prod\limits_{k = r}^{n - 1}(\delta_k' x - \gamma_k' y),
$$
where we now have
$$
(\gamma_j')^2 + (\delta_j')^2 = 1 \quad \text{and} \quad \left|\gamma_k'\right|^2 + \left|\delta_k'\right|^2 = 1.
$$
In particular, notice that for $j = 0, 1, \ldots, r - 1$ each point $(\gamma_j', \delta_j')$ lies on the unit circle, so $(\gamma_j', \delta_j') = (\cos \alpha_j, \sin \alpha_j)$ for some $\alpha_j \in \mathbb R$. Since
\begin{align*}
F(\cos \theta, \sin \theta)
& = h_F \prod\limits_{j = 0}^{r - 1}\sin\left(\theta - \alpha_j\right)\prod\limits_{k = r}^{n - 1}(\delta_k' \cos \theta - \gamma_k' \sin \theta),
\end{align*}
we can use the formula (\ref{eq:polar-formula}) to compute $A_F$:
\begin{align*}
A_F
& = \frac{1}{2}\int\limits_0^{2\pi}\frac{d\theta}{|F(\cos \theta, \sin \theta)|^{2/n}}\\
& = \frac{1}{2h_F^{2/n}}\int\limits_0^{2\pi}\sqrt[n]{\prod\limits_{j = 0}^{r - 1}\csc^2\left(\theta - \alpha_j\right)} \cdot \sqrt[n]{\prod\limits_{k = r}^{n - 1}|\delta_k' \cos \theta - \gamma_k' \sin \theta|^{-2}}\ d\theta.
\end{align*}
Since the integrand has period $\pi$, we conclude that
$$
h_F^{2/n}A_F = \int\limits_0^{\pi}\sqrt[n]{\prod\limits_{j = 0}^{r - 1}\csc^2\left(\theta - \alpha_j\right)} \cdot \sqrt[n]{\prod\limits_{k = r}^{n - 1}|\delta_k' \cos \theta - \gamma_k' \sin \theta|^{-2}}\ d\theta.
$$
By Jensen's inequality (\ref{eq:jensen}) with $\varphi(x) = \log x$,
\begin{align*}
\log\left(\frac{h_F^{2/n}A_F}{\pi}\right) 
& \geq \frac{1}{\pi}
    \int\limits_0^{\pi}\log\left(\sqrt[n]{\prod\limits_{j = 0}^{r - 1}\csc^2\left(\theta - \alpha_j\right)} \cdot \sqrt[n]{\prod\limits_{k = r}^{n - 1}|\delta_k' \cos \theta - \gamma_k' \sin \theta|^{-2}}\right)\ d\theta\\
& = \frac{1}{n\pi}\sum\limits_{j = 0}^{r - 1}\int\limits_0^\pi\log(\csc^2(\theta - \alpha_j))d\theta + \frac{1}{n\pi}\sum\limits_{k = r}^{n - 1}\int\limits_0^\pi\log\left(|\delta_k' \cos \theta - \gamma_k'\sin \theta|^{-2}\right)d\theta\\
& = \frac{r}{n\pi}\int\limits_0^\pi\log(\csc^2(\theta))d\theta + \frac{1}{n\pi}\sum\limits_{k = r}^{n - 1}\int\limits_0^\pi\log\left(|\delta_k' \cos \theta - \gamma_k'\sin\theta|^{-2}\right)d\theta\\
\text{Lemma \ref{lem:logsin}} \rightarrow & = \frac{r}{n}\log 4 + \frac{1}{n\pi}I,
\end{align*}
where
$$
I = \sum\limits_{k = r}^{n - 1}\int\limits_0^\pi\log\left(|\delta_k' \cos \theta - \gamma_k'\sin \theta|^{-2}\right)d\theta.
$$


It remains to prove that $I \geq (n - r)\pi \log 2$. Since $\left|\gamma_k'\right|^2 + \left|\delta_k'\right|^2 = 1$, there exist real numbers $\beta_k$ such that
$$
|\gamma_k'|=\cos \beta_k \quad \text{and} \quad |\delta_k'|=\sin \beta_k.
$$
%
%
%
Notice that 
\begin{align*}
|\delta_k' \cos \theta - \gamma_k' \sin \theta| \cdot |\delta_k' \cos (\pi-\theta) - \gamma_k' \sin (\pi-\theta)|
& = |(\delta_k' \cos \theta)^2 - (\gamma_k' \sin \theta)^2|\\
& \leq \left|\delta_k' \cos \theta\right|^2 + \left|\gamma_k' \sin \theta\right|^2\\
& = (\sin\beta_k\cos\theta)^2+(\cos\beta_k\sin\theta)^2,
\end{align*}
so
$$
(|\delta_k' \cos \theta - \gamma_k' \sin \theta||\delta_k' \cos (\pi-\theta) - \gamma_k' \sin (\pi-\theta)|)^{-2}\geq ((\sin\beta_k\cos\theta)^2+(\cos\beta_k\sin\theta)^2)^{-2}.
$$
Consequently,
\small
\begin{align*}
2I
& = \sum\limits_{k = r}^{n - 1}\int\limits_0^\pi\log\left(|\delta_k' \cos \theta - \gamma_k'\sin \theta|^{-2}\right)d\theta + \sum\limits_{k = r}^{n - 1}\int\limits_0^\pi\log\left(|\delta_k' \cos(\pi - \theta) - \gamma_k'\sin(\pi - \theta)|^{-2}\right)d\theta\\
& = \sum\limits_{k = r}^{n - 1}\int\limits_0^\pi\log\left(|\delta_k' \cos \theta - \gamma_k'\sin \theta| \cdot |\delta_k' \cos(\pi - \theta) - \gamma_k'\sin(\pi - \theta)|\right)^{-2} d\theta\\
& \geq \sum\limits_{k = r}^{n - 1}\int\limits_0^\pi\log((\sin\beta_k\cos\theta)^2+(\cos\beta_k\sin\theta)^2)^{-2} d\theta\\
\text{Lemma \ref{lem:bound}} \rightarrow & \geq 2(n - r)\pi \log 2.
\end{align*}
\normalsize

In summary, we proved that
$$
\log\left(\frac{h_F^{2/n}A_F}{\pi}\right) \geq \frac{r}{n}\log 4 + \frac{n - r}{n}\log 2.
$$
Exponentiation on both sides yields the desired inequality $h_F^{2/n}A_F\geq \left(2^{1+(r/n)}\right)\pi$.

\section{The Family $F_{n, r}$} \label{sec:Gnr}

In this section we derive formulas for the height and the area of the fundamental region of a binary form $F_{n, r}$ defined in (\ref{eq:Fnr}). By definition, the height of $F_{n, r}$ is given by $h_{F_{n, r}} = 2^{(n - r)/2}$. To compute $A_{F_{n, r}}$, notice that for $r \geq 1$ we have
\begin{align*}
\left|F_{n, r}(\cos \theta, \sin \theta)\right|
& = \left|(\cos \theta - i\sin \theta)^{n - r}\prod\limits_{k = 1}^r\left(\cos \theta\sin\frac{k\pi}{r} - \sin \theta\cos\frac{k\pi}{r}\right)\right|\\
& = \left|\prod\limits_{k = 1}^r\sin\left(\theta - \frac{k\pi}{r}\right)\right|
 = \left|2^{-(r - 1)}\sin(r\theta)\right|,
\end{align*}
where the last equality follows from the identity $\sin(r\theta) = 2^{r - 1}\prod_{k = 1}^r\sin\left(\frac{k\pi}{r} - \theta\right)$ (see, for example, \cite[Section 2]{mosunov1}). By (\ref{eq:polar-formula}),

\begin{align*}
A_{F_{n, r}}
& = \frac{1}{2}\int\limits_0^{2\pi}\sqrt[n]{4^{r - 1}\csc^2(r\theta)}d\theta\\
& = 2^{2(r - 1)/n - 1}\int\limits_0^{2\pi}\sqrt[n]{\csc^2(r\theta)}d\theta\\
& = \frac{2^{2(r - 1)/n - 1}}{r}\int\limits_0^{2\pi r}\sqrt[n]{\csc^2\theta}d\theta\\
& = \frac{2^{2(r - 1)/n - 1}}{r} \cdot 4r\int\limits_0^{\pi/2}(\sin \theta)^{-\frac{2}{n}}d\theta.
\end{align*}
Now, for $n = 1, 2$ the above integral diverges, while for $n \geq 3$ it follows from (\ref{eq:beta}) that
$$
A_{F_{n, r}} = 2^{2(r - 1)/n} \cdot 2\int\limits_0^{\pi/2}(\sin \theta)^{-\frac{2}{n}}d\theta = 2^{2(r - 1)/n}B\left(\frac{1}{2} - \frac{1}{n}, \frac{1}{2}\right).
$$
Thus,
$$
h_{F_{n, r}}^{2/n}A_{F_{n, r}} = 
\begin{cases}
2\pi & \text{if $r = 0$,}\\
\infty & \text{if $n = 1, 2$ and $r > 0$,}\\
2^{1 + (r - 2)/n}B\left(\frac{1}{2} - \frac{1}{n}, \frac{1}{2}\right) & \text{if $n > 2$ and $r > 0$.}
\end{cases}
$$

\bibliography{fang-mosunov-lower-bound}
\bibliographystyle{plain}

\end{document}